\documentclass[12pt,dvips]{report}

\usepackage{amsmath}
\usepackage{amssymb}
\usepackage{amsthm}
\usepackage{amscd}
\usepackage{pictexwd,dcpic}
\usepackage{slashed}

\theoremstyle{plain}
\newtheorem{theorem}{Theorem}[section]
\newtheorem{lemma}[theorem]{Lemma}
\newtheorem{proposition}[theorem]{Proposition}
\newtheorem{corollary}[theorem]{Corollary}

\theoremstyle{definition}
\newtheorem{definition}[theorem]{Definition}

\theoremstyle{remark}
\newtheorem{example}[theorem]{Example}
\newtheorem{remark}[theorem]{Remark}

\begin{document}

\thispagestyle{empty}
\pagenumbering{roman}
\begin{center}

{\Large
D-branes and $K$-homology
}

\vfill

Bei Jia

\vfill

Thesis submitted to the Faculty of the \\
Virginia Polytechnic Institute and State University \\
in partial fulfillment of the requirements for the degree of

\vfill

Master of Science \\
in \\
Mathematics

\vfill

Peter E. Haskell, Chair \\
William J. Floyd \\
Peter A. Linnell \\
Eric R. Sharpe

\vfill

April 19, 2013 \\
Blacksburg, Virginia

\vfill

Keywords:$K$-homology, $K$-theory, D-brane, String theory
\\
Copyright 2013, Bei Jia

\end{center}

\pagebreak

\thispagestyle{empty}
\begin{center}

{\large D-branes and $K$-homology}

\vfill

Bei Jia

\vfill

(ABSTRACT)

\vfill

\end{center}

In this thesis the close relationship between the topological $K$-homology group of the spacetime manifold $X$ of string theory and D-branes in string theory is examined. An element of the $K$-homology group is given by an equivalence class of $K$-cycles $[M,E,\phi]$, where $M$ is a closed spin$^c$ manifold, $E$ is a complex vector bundle over $M$, and $\phi: M\rightarrow X$ is a continuous map. It is proposed that a $K$-cycle $[M,E,\phi]$ represents a D-brane configuration wrapping the subspace $\phi(M)$. As a consequence, the $K$-homology element defined by $[M,E,\phi]$ represents a class of D-brane configurations that have the same physical charge. Furthermore, the $K$-cycle representation of D-branes resembles the modern way of characterizing fundamental strings, in which the strings are represented as two-dimensional surfaces with maps into the spacetime manifold. This classification of D-branes also suggests the possibility of physically interpreting D-branes wrapping singular subspaces of spacetime, enlarging the known types of singularities that string theory can cope with.

\vfill

\pagebreak

\chapter*{Acknowledgments}
I would like to thank my advisor, Professor Haskell, for all the help he gave me during my Master's study. I have learned a tremendous amount of mathematics from him. He introduced the topics of $K$-theory and $K$-homology to me, as well as many other relevant materials. He even provided many of the technical details that are used in this thesis. I am especially thankful that he listened to many of my problems, both academic and personal, with great patience, and offered many valuable advice. I consider myself lucky to have Professor Haskell as my advisor.

I would also like to thank my committee members: Professor Floyd, Professor Linnell, and Professor Sharpe, for their help and support. Especially, I would like to thank Professor Sharpe, who is my Ph.D. advisor in physics, for his support of me in pursuing this Master's degree in mathematics simultaneously with my Ph.D. study in physics.

Finally, I would like to thank my parents and my wife for their support and understanding. I could not have done anything without them.

\tableofcontents
\pagebreak

\pagenumbering{arabic}
\pagestyle{myheadings}

\chapter{Introduction}

This thesis is devoted to a discussion of the relationship between D-branes and $K$-homology. The geometric nature of topological $K$-homology, which provides a direct classification of D-branes in string theory, captures the physical information of D-branes.

\subsection*{D-branes in String Theory}
The basic idea of string theory is to abandon the concept of point-like particles, which causes much theoretical trouble for physicists. Instead, string theory assumes that the fundamental building blocks of our world are tiny one-dimensional strings, which are so small that they look like particles at low energies. The large number of fundamental particles that have been observed so far can be interpreted as different quantum states of strings. One of the biggest achievements of string theory is that it contains gravity, and thus serves as a quantum theory of gravity.

There are two types of fundamental strings: open strings and closed strings. The trajectory of a string in spacetime is a two-dimensional surface, called the string worldsheet, on which a two-dimensional (supersymmetric) conformal field theory can be defined to characterize the dynamics of this string. The worldsheet of an open string is a two-dimensional disk with holes inside it, while the worldsheet of a closed string is a Riemann surface. The more advanced way of characterizing the nature of fundamental strings is to think of \emph{a string as characterized by a two-dimensional surface (usually with more structure) with a map from this surface into the spacetime manifold, with the image of this map being identified as the worldsheet of this string}. Later we will propose a similar picture for D-branes.

It took string theorists many years to realize that string theory is not just about strings. There are various extended objects that live in the ten dimensional spacetime manifold of string theory; they all have very interesting dynamics of their own. A large class of such extended objects in string theory is called the class of D-branes. Roughly speaking, a D-brane is an extended dynamical object in spacetime on which open strings can end. As such, it provides Dirichlet boundary conditions for open strings, hence the name. The importance of D-branes' physics was not realized by mainstream researchers until \cite{Polchinski-brane}. This realization is one of the many inspiring aspects of the second string revolution, which happened around the middle of the 1990s. Since then, the study of D-branes has become a fundamental aspect of string theory research and has led to tremendous development in many directions, such as homological mirror symmetry on the mathematics side, and the celebrated AdS/CFT conjecture on the physics side.

It is important to emphasize that string theory, with its many contributions to both physics and mathematics, is still not a completely formulated theory. Various aspects of string theory have been discovered, yet a fundamental formulation is still missing. Even the subject of the dynamics of a free string propagating in a general background spacetime is not fully understood. As a consequence, there is not a fundamental definition of D-branes that captures all of the information about them.

D-branes, as extended dynamical objects, can wrap submanifolds of spacetime. These submanifolds represent the trajectories of D-branes in spacetime. They are called D-brane worldvolumes. At low energies, the dynamics of a D-brane can be approximated by a (supersymmetric) gauge theory living on its worldvolume. Mathematically, this means that we have a vector bundle (with connection) over that submanifold; physically this vector bundle describes the associated gauge interaction. One of the major tasks of D-brane research is to find a more fundamental way of characterizing D-branes and their dynamics, either through physical constructions or mathematical constructions.

\subsection*{$K$-theory}
Mathematically, the study of $K$-theory originated as the study of a ring generated by equivalence classes of vector bundles over a topological space or scheme. More precisely, what we will discuss is called topological $K$-theory. Because it is the theory directly used in string theory, we will simply refer to it as $K$-theory. Abstractly, it is a contravariant functor from the category of topological spaces to the category of commutative rings. Let $f: E\rightarrow F$ be a homomorphism of complex vector bundles over a smooth manifold $X$. The pair $(E,F)$ determines an element in $K^0(X)$, by the equivalence relationship $(E,F)\sim (E\oplus W, F\oplus W)$ for any complex vector bundle $W$ over $X$.

The appearance of vector bundles over D-branes suggests the possibility of representing D-branes as $K$-cocyles. Let's consider D-branes in type IIB string theory as an example. Consider the physical configuration of D9-branes and anti-D9-branes in type IIB string theory: let $E$ be the gauge bundle of the D9-branes, while $F$ is the gauge bundle of the anti-D9-branes. Physically, there is an annihilation process, called tachyon condensation, between D-branes and anti-D-branes with isomorphic gauge bundles. If we add to the configuration any D9-anti-D9-brane pairs with the same gauge bundle W, then physically these added branes cancel each other via tachyon condensation over the entire spacetime, leaving us with the original branes. In other words, tachyon condensation over the entire spacetime is the same as the equivalence relation defining $K$-theory elements. Therefore, D9-anti-D9-branes can be represented by $K$-theory elements of the spacetime manifold.

A further exploration of tachyon condensation leads to a $K$-theoretic interpretation of lower dimensional D-branes, starting from the above construction. Suppose there is a vector bundle map $f: E\rightarrow F$ that is an isomorphism outside a subspace $S$ of $X$. Physically, the triple $(E,F,f)$ represents a D-brane configuration wrapping the subspace $S$, ``collapsed'' from the D9-brane-anti-D9-branes represented by $(E,F)$; the map $f$ is precisely the tachyon condensation, because it is an isomorphism outside $S$.  This example, in which $S$ has dimension lower than ten, suggests that there should be a $K$-theoretic interpretation of some properties of D-branes of all dimensions.

What is precisely the information about D-branes that is encoded in this $K$-theoretic interpretation of D-branes? D-branes are charged under the Ramond-Ramond form fields from the string spectra, and these charges are preserved under tachyon condensation. Therefore, \emph{the information about D-branes that can be represented by the $K$-theory elements of spacetime is precisely these charges} \cite{minasian-moore-ktheory, witten-ktheory, horava-ktheory}:
\begin{itemize}
\item Type IIA string theory $\leftrightarrow \tilde{K}^1(X)$
\item Type IIB string theory $\leftrightarrow \tilde{K}^0(X)$
\item Type I string theory $\leftrightarrow \widetilde{KO}(X)$
\end{itemize}

There are many important aspects of this $K$-theoretic interpretation of D-branes. One excellent example is the Bott periodicity of complex $K$-groups \cite{bott-peorid}, namely there are exactly two $K$-groups $K^0(X)$ and $K^1(X)$. This mathematical property corresponds precisely to the fact that there are only two kinds of type II string theories: type IIA string theory and type IIB string theory. In the case of type I string theory, Bott periodicity of real $K$-theory groups provides of the existence of many new types of D-branes in type I string theory by the Bott periodicity of real $K$-groups. Another example is that $K$-theory requires D-branes to wrap on manifolds with spin$^c$ structures, which coincides with a requirement arising from anomaly computation.

\subsection*{$K$-homology}
The information contained in a D-brane configuration is rather geometric in nature: vector bundles over subspaces. It is then natrual to look for a geometric way of representing D-branes using some versions of homology cycles, rather than the $K$-cocycles in $K$-theory, which is a generalized cohomology theory. One can formulate the homological dual of $K$-theory, which is usually called $K$-homology. There are various definitions of $K$-homology in different situations; the topological definition of the $K$-homology of a topological space $X$ is given in \cite{baum-douglas} as certain equivalence classes of $K$-cycles, which are triples of the form $[M,E,f]$ satisfying
\begin{itemize}
    \item $M$ is a compact spin$^c$ manifold
\item $E$ is a complex vector bundle over $M$
\item $f: M \rightarrow X$ is a continuous map
\end{itemize}
This is the dual of the $K$-theory of complex vector bundles on $X$.

The relationship between D-branes and $K$-homology has been explored previously \cite{harvey-moore-khomology, asakawa-khomology, szabo, reis-szabo-khomology, baum-khomology-brane}. It will be the main topic of this thesis. In this thesis, we propose that \emph{the $K$-homology groups of spacetime provide a direct classification of D-branes, in the sense that the image $\phi(M)$ of a $K$-cycle $[M,E,f]$ should be interpreted as the worldvolume of a D-brane configuration}. This way the geometric structure of $K$-cycles expresses the physical structure of D-branes, in contrast to the information about D-brane charges included in $K$-theory.

The $K$-theory equivalence relation is naturally interpreted as tachyon condensation from the physics point of view. Correspondingly, this process is preserved from the $K$-homology point of view: the equivalence relations that define an element of the $K$-homology group correspond precisely to tachyon condensation, too. The details can be found in Section 3 of Chapter 3.

We would like to emphasize one property of our proposal, namely it is very much like the modern way of characterizing fundamental strings, which are characterized by two-dimensional surfaces with maps into spacetime. Our proposal treats the worldvolumes of D-branes and the worldsheets of fundamental strings in a analogous fashion.

Our proposal could lead in many other directions. On one hand, it provides a concrete application of the idea of $K$-homology; one might hope to further expand the current knowledge about D-branes using relevant mathematics. On the other hand, further research on the properties of D-branes in physics might provide more insights about $K$-homology itself, as usually happens in the field of physical mathematics.

This thesis is organized as follows. In Chapter 2, we introduce topological $K$-theory and analyze some of its properties. Then we discuss how one can use the $K$-theory group of spacetime to classify the charges of D-branes in type II string theories. In Chapter 3, we move on to introduce topological $K$-homology. We propose a relationship between the $K$-homology group of spacetime and classes of D-branes in type II string theories. Furthermore, we discuss how one can characterize type II D-branes using $K$-cycles, whose equivalence classes define the $K$-homology group of spacetime. Finally, in Chapter 4 we present the conclusion and further discussion.

\pagebreak

\chapter{$K$-theory and D-branes}

\section{Vector Bundles}
Topological $K$-theory is the group of certain equivalence classes of vector bundles. First we briefly introduce the notion of a vector bundle, and some of its relevant properties. We assume that all maps are continuous, unless stated otherwise.

\begin{definition}
Let $X$ be a topological space. A \emph{vector bundle} over $X$ is a surjective map $\pi : E\rightarrow X$, such that
\begin{itemize}
\item $E_x=\pi^{-1}(x)$ is a vector space over some field $k$ for all $ x\in X$;

\item There exists an open cover $\{U_{\alpha}\}$ of $X$ with associated homeomorphisms
 \begin{equation*}
 h_{\alpha}: \pi^{-1}(U_{\alpha})\rightarrow U_{\alpha}\times k^n
 \end{equation*}
 where $n\in \mathbb{Z}$, such that the restriction
 \begin{equation*}
 h_{\alpha}|_x: \pi^{-1}(x)\rightarrow \{x\}\times k^n
 \end{equation*}
 is an vector space isomorphism for all $ x\in U_{\alpha}$.
\end{itemize}
$E$ is called the \emph{total space}, $E_x$ is called the \emph{fiber}, $X$ is called the \emph{base space}, while $(\{U_{\alpha}\},\{h_{\alpha}\})$ is called the \emph{local trivialization}.
\end{definition}

Note that the dimension $n$ of $E_x$ is only a constant locally on each connected component of $X$. If it is a constant over $X$, we say the vector bundle has rank $n$. We usually also refer to the total space $E$ when we talk about vector bundle.

\begin{example}
The Cartesian product $X\times k^n$ is a vector bundle, called the trivial bundle for the obvious reason.
A slightly nontrivial example is the M\"{o}bius strip, which is a real line bundle (a real vector bundle of rank 1) over $S^1$.
\end{example}

\begin{definition}
A \emph{global section} of a vector bundle $\pi : E\rightarrow X$ is a map $s : X\rightarrow E$ such that $\pi\circ s(x)=x$ for all $ x\in X$. A \emph{local section} over an open subset $U\subseteq X$ is a map $s_U : U\rightarrow E$ such that $\pi\circ s_U(x)=x$ for all $ x\in U$.
\end{definition}

\begin{example}
Let $M$ be a smooth manifold of dimension $n$. The tangent bundle of $M$, denoted by $TM$, is the disjoint union of the tangent spaces of $M$, which is a vector bundle of rank $n$. A global section of $TM$ is simply a vector field on $M$.
\end{example}

\begin{definition}
Let $\pi: E\rightarrow X$ and $\rho: F\rightarrow Y$ be two vector bundles over the same field. Then a \emph{vector bundle map} is a pair of maps $f: E\rightarrow F$ and $g: X\rightarrow Y$, such that
\begin{itemize}
\item $\rho\circ f = g\circ\pi$, i.e. the following diagram commutes
\begin{equation*}
\begin{CD}
E @>f>> F\\
@VV\pi V @VV\rho V\\
X @>g>> Y
\end{CD}
\end{equation*}
\item $f_x: E_x\rightarrow F_{g(x)}$ is a linear map between vector spaces for all $ x\in X$.
\end{itemize}
\end{definition}

Clearly, with this definition of vector bundle map, we can assign a category structure to the collection of all vector bundles: the objects being the vector bundles, and the morphisms being the vector bundle maps.

\begin{definition}
Given a vector bundle $\pi:E\rightarrow Y$, and a map $f:X\rightarrow Y$ between two topological spaces $X$ and $Y$, we define the \emph{pullback bundle} $f^*E\rightarrow X$ as
\begin{equation*}
f^*E=\{ (x,e)\in X\times E | f(x)=\pi(e) \}.
\end{equation*}
\end{definition}

\begin{remark}
Note that we can perform linear algebraic operations, such as direct sum and tensor product, on each fiber $E_x$, which is a vector space. This leads to the following operations on vector bundles:
\begin{itemize}
\item The direct sum (Whitney sum) of vector bundles $E\rightarrow X$ and $F\rightarrow X$ is a vector bundle $E\oplus F\rightarrow X$, whose fiber is $E_x\oplus F_x$.
\item The tensor product of vector bundles $E\rightarrow X$ and $F\rightarrow X$ is a vector bundle $E\otimes F\rightarrow X$, whose fiber is $E_x\otimes F_x$.
\end{itemize}
\end{remark}

Let $Vect(X)$ denote the set of all vector bundles over $X$. Clearly, $Vect(X)$ is an abelian semigroup under the direct sum operation. Our goal in the next section is to construct from $Vect(X)$ an abelian group, which will be the $K$-theory group of the space $X$. From now on, we will mostly assume that we work with $k=\mathbb{C}$, unless otherwise stated.

\begin{lemma} \label{complementary-bundle}
For any vector bundle $E$ over a compact Hausdorff space $X$, there exists a vector bundle $F\rightarrow X$ such that $E\oplus F \cong \varepsilon^n$, where $\varepsilon^n = X\times \mathbb{C}^n$ for some natural number $n$.
\end{lemma}

In general, one can think of bundles whose fibers are not necessarily vector spaces. This leads to the following definition of a fiber bundle, analogous to the definition of vector bundles
\begin{definition}
Let $X$ be a topological space. A \emph{fiber bundle} over $X$ is a surjective map $\pi : E\rightarrow X$, such that
\begin{itemize}
\item $F=\pi^{-1}(x)$ is a topological space for all $ x\in X$;

\item there exists  an open cover $\{U_{\alpha}\}$ of $X$ with associated homeomorphisms
 \begin{equation*}
 h_{\alpha}: \pi^{-1}(U_{\alpha})\rightarrow U_{\alpha}\times F
 \end{equation*}
 such that the restriction
 \begin{equation*}
 h_{\alpha}|_x: \pi^{-1}(x)\rightarrow \{x\}\times F
 \end{equation*}
 is a homeomorphism for all $ x\in U_{\alpha}$.
\end{itemize}
$E$ is called the \emph{total space}, $F$ is called the \emph{fiber}, $X$ is called the \emph{base space}, while $(\{U_{\alpha}\},\{h_{\alpha}\})$ is called the the \emph{local trivialization}.
\end{definition}

\begin{definition}
Let $G$ be a topological group. A \emph{principal $G$-bundle} is a fiber bundle $P\rightarrow X$ with a continuous right $G$ action on $P$ such that $G$ preserves the fibers of $P$ and acts freely and transitively on them. Then one can define the \emph{associated vector bundle} of $P$ with respect to a representation $V$ of $G$ as
\begin{equation}
E=P\times_G V := P\times V/(p,v)\sim (pg, g^{-1}v), \quad \text{for all} \ p\in P, v\in V, g\in G
\end{equation}
\end{definition}

In both physics and mathematics, use of the concept of spinors is in many cases inevitable. A spinor is a section of a \emph{spinor bundle}, which is the associated vector bundle of the principal bundle $P$ in the following definition:

\begin{definition}
Let $\pi:E\rightarrow X$ be a real vector bundle of rank $n$, with its frame bundle denoted as $P_{\text{SO}(n)}$. A \emph{spin structure} on $E$ is a principal Spin$(n)$-bundle $P_{\text{Spin}(n)}\rightarrow X$, together with a equivariant double covering map $\phi:P_{\text{Spin}(n)}\rightarrow P_{\text{SO}(n)}$, i.e. $\phi(pg)=\phi(p)\psi(g)$ for all $ p\in P_{\text{Spin}(n)}, g\in$ Spin$(n)$, where $\psi:$ Spin$(n)\rightarrow$SO$(n)$ is the double covering homomorphism.
\end{definition}

Similarly, one can define the spin$^c$ structure on a vector bundle, using the Spin$^c(n)$ group which satisfies the short exact sequence
\begin{equation*}
1\rightarrow \mathbb{Z}_2\rightarrow \text{Spin}^c(n)\stackrel{\theta}{\rightarrow} \text{SO}(n)\times \text{U}(1)\rightarrow 1
\end{equation*}
Then as in the case of spin structure, we can define
\begin{definition}
Let $\pi:E\rightarrow X$ be a real vector bundle of rank $n$, with its frame bundle denoted as $P_{\text{SO}(n)}$. A \emph{spin$^c$ structure} on $E$ is a principal U$(1)$-bundle $P_{\text{U}(1)}\rightarrow X$, and a principal Spin$^c(n)$-bundle $P_{\text{Spin}^c(n)}\rightarrow X$ together with a equivariant bundle map $\phi:P_{\text{Spin}^c(n)}\rightarrow P_{\text{SO}(n)}\times P_{\text{U}(1)}$, i.e. $\phi(pg)=\phi(p)\theta(g)$ for all $ p\in P_{\text{Spin}^c(n)}, g\in$ Spin$^c(n)$, where $\theta:$ Spin$^c(n)\rightarrow$ SO$(n)\times$U$(1)$ is the defining homomorphism in the above short exact sequence.
\end{definition}

\begin{definition}
If there is a spin (spin$^c$) structure on the tangent bundle of an oriented Riemannian manifold, then this manifold is called a \emph{spin (spin$^c$) manifold}.
\end{definition}

\begin{remark}
Let $X$ be a spin$^c$ manifold with dimension $n$. If $n$ is even, then there is only one fundamental spinor bundle $S(X)$ which splits $S(X)=S^+(X)\oplus S^-(X)$, which in physics represents the two chiralities of fermions in even dimensions. If $n$ is odd, there is again only one fundamental spinor bundle $S(X)$ which does not split.
\end{remark}

\section{Characteristic Classes}
Characteristic classes, which are cohomology classes, are a necessary (but not sufficient) measure of  how a vector bundle is different from a trivial bundle. The natural characteristic classes of a complex vector bundle are the Chern classes, which we define now, using the axiomatic approach:

\begin{definition}
Let $E$ be an element in $Vect(X)$ of a topological space $X$. Then there exists a unique sequence of maps $c_i: Vect(X)\rightarrow H^{2i}(X;\mathbb{Z})$, such that\\
(1) $c_i(f^*(E))=f^*(c_iE)$ for all pullbacks $f^*(E)$.\\
(2) Let $c(E)=1+c_1+c_2+...\in H^*(X,\mathbb{Z})$. Then $c(E\oplus F)=c(E)\cup c(F)$.\\
(3) $c_i(E)=0$ for all $i$ larger than the rank of $E$.\\
(4) Let $L$ be the canonical line bundle over $\mathbb{C}P^{1}$. Then $c_1(L)$ is the generator of $H^2(\mathbb{C}P^{1},\mathbb{Z})$ specified in advance.\\
$c(E)$ is called the \emph{total Chern class} of $E$, while $c_i(E)$ is called the \emph{$i$th Chern class} of $E$.
\end{definition}

Note that the last condition is a normalization condition as well as a nontriviality condition. There are several other definitions of Chern classes, some of which are equivalent to the above axiomatic one, such as using the pullback of the generators of the cohomology of the classifying space of $U(n)$ bundles, while others are weaker than the above one, such as the Chern-Weil approach.

\begin{example}
Let's consider a trivial bundle $E=X\times \mathbb{C}^n$. Let $f:X\rightarrow \text{point}$ be the map to a point, then $E=f^*(\text{point}\times \mathbb{C}^n)$. Since $H^i(\text{point},\mathbb{Z})=0$ for all $i>0$, condition (1) means $c_i(E)=0$ for all $i>0$. Thus trivial bundles have trivial Chern classes. However, a nontrivial bundle might also have trivial Chern classes.
\end{example}

\begin{proposition} \textbf{(Splitting principle)}
Let $E\rightarrow X$ be a complex vector bundle of rank $n$. Then there is a space $F(E)$ with a map $p:F(E)\rightarrow X$ such that
\begin{equation*}
p^*(E)=L_1\oplus L_2\oplus ... \oplus L_n
\end{equation*}
where $L_i\rightarrow F(E)$ are line bundles. Furthermore, the induced homomorphism on cohomology $p^*: H^*(X,\mathbb{Z})\rightarrow H^*(F(E),\mathbb{Z})$ is injective.
\end{proposition}

One can show that $F(E)$ is a fiber bundle over $X$, whose fiber is the \emph{flag} of the fiber of $E$, namely $V_1\subset V_2\subset...\subset V_n$ where $V_i$ is a dimension $i$ linear subspace of the fiber of $E$. As such, $F(E)$ is called the \emph{flag bundle} associated to $E$.

\begin{definition}
Denote $c_1(L_i)$ as $x_i$. These are called the \emph{Chern roots} of $E$.
\end{definition}

\begin{corollary}
\begin{equation*}
p^*c(E)=1+\sum_{1\leqslant i\leqslant n} x_i +\sum_{1\leqslant i<j\leqslant n} x_ix_j+\cdots
\end{equation*}
i.e. $p^*c_i(E)$ is the $k$th elementary symmetric product of the Chern roots of $E$.
\end{corollary}

\begin{remark}
Each elementary symmetric product of the Chern roots of $E$ determines a Chern class of $p^*E$ as element in $H^*(F(E),\mathbb{Z})$. Since $p^*: H^*(X,\mathbb{Z})\rightarrow H^*(F(E),\mathbb{Z})$ is an injective ring homomorphism, we see each elementary symmetric product of the Chern roots of $E$ also determines an element in $H^*(X,\mathbb{Z})$. Therefore, schematically we can write $c(E)=1+\sum_i x_i +\sum_{i<j} x_ix_j+\cdots$.
\end{remark}

\begin{definition}
The \emph{Chern character} of a vector bundle $E\rightarrow X$ is defined as
\begin{equation*}
ch^{\bullet}(E)=\sum_i e^{x_i}=\sum_i\left(\sum_{k=0}^{\infty}\frac{x_i^k}{k!}\right)\in H^*(X,\mathbb{Q})
\end{equation*}
\end{definition}

\begin{lemma}
Let $E,F$ be two vector bundles over $X$. Then
\begin{equation*}
\begin{split}
& ch^{\bullet}(E\oplus F)=ch^{\bullet}(E)+ch^{\bullet}(F)\\
& ch^{\bullet}(E\otimes F)=ch^{\bullet}(E)\cup ch^{\bullet}(F)
\end{split}
\end{equation*}
\end{lemma}

\begin{definition}
Let $E\rightarrow X$ be a complex vector bundle. Then the \emph{Todd class} of $E$ is
\begin{equation*}
Td(E)=\prod_i \frac{x_i}{1-e^{-x_i}}\in H^*(X,\mathbb{Q})
\end{equation*}
The \emph{$\hat{A}$ genus} of $E$ is
\begin{equation*}
\hat{A}(E)=\prod_i \frac{x_i/2}{\sinh(x_i/2)}\in H^*(X,\mathbb{Q})
\end{equation*}
\end{definition}

\begin{lemma}
\begin{equation*}
Td(E)=e^{c_1(E)/2}\hat{A}(E)
\end{equation*}
\end{lemma}
\begin{proof}
From the splitting principle, we have $c_1(E)=\sum_i x_i$. Then by definition
\begin{equation*}
\begin{split}
e^{c_1(E)/2}\hat{A}(E) &= e^{\sum_i x_i/2} \prod_i \frac{x_i/2}{\sinh(x_i/2)}\\
&= \prod_i \frac{e^{x_i/2}x_i/2}{\sinh(x_i/2)}\\
&= \prod_i \frac{e^{x_i/2}x_i}{e^{x_i/2}-e^{-x_i/2}}\\
&= \prod_i \frac{x_i}{1-e^{-x_i}}\\
&= Td(E)
\end{split}
\end{equation*}
\end{proof}

For real vector bundles, one can similarly define the Stiefel-Whitney class via the axiomatic approach:

\begin{definition}
Let $E$ be a real vector bundle over a topological space $X$. Then there exists a unique sequence of maps $w_i: Vect(X)\rightarrow H^{i}(X;\mathbb{Z}_2)$, such that\\
(1) $w_i(f^*(E))=f^*(w_iE)$ for all pullback $f^*(E)$.\\
(2) Let $w(E)=1+c_1+c_2+...\in H^*(X,\mathbb{Z}_2)$. Then $w(E\oplus F)=w(E)\cup w(F)$.\\
(3) $w_i(E)=0$ for all $i$ larger than the rank of $E$.\\
(4) Let $L$ be the canonical line bundle over $\mathbb{R}P^{1}$. Then $w_1(L)$ is the generator of $H^1(\mathbb{R}P^{1},\mathbb{Z}_2)$ specified in advance.\\
$w(E)$ is called the \emph{total Stiefel-Whitney class} of $E$, while $w_i(E)$ is called the \emph{$i$th Stiefel-Whitney class} of $E$.
\end{definition}

\begin{remark}
The above results and properties of Chern classes of complex vector bundles have similar versions for Stiefel-Whitney classes of real vector bundles.
\end{remark}

\begin{proposition}
A real vector bundle is orientable if and only if its first Stiefel-Whitney class $w_1=0$. A spin structure on a vector bundle $E$ exists if and only if $E$ is orientable and its second Stiefel-Whitney class $w_2=0$. A spin$^c$ structure on a vector bundle $E$ exists if and only if $E$ is orientable and the second Stiefel-Whitney class of $E$ is in the image of the natural map $H^2(X,\mathbb{Z})\rightarrow H^2(X,\mathbb{Z}_2)$.
\end{proposition}

There is another type of characteristic class for real vector bundles:
\begin{definition}
Let $E\rightarrow X$ be a real vector bundle. Then the \emph{$i$th Pontryagin class} of $E$ is defined as
\begin{equation*}
p_i(E)=(-1)^i c_{2i}(E \otimes \mathbb{C})\in H^{4i}(X,\mathbb{Z})
\end{equation*}
where $E \otimes \mathbb{C}$ is the complexification of $E$.
\end{definition}
Notice that the Pontryagin classes are defined via the Chern classes of $E \otimes \mathbb{C}$, thus given by the Chern roots $x_i$ of $E \otimes \mathbb{C}$, i.e. $p_i$ is given by the elementary symmetric function of $x_i^2$. Then one can define the $\hat{A}$ genus of $E$ as for complex vector bundles
\begin{equation*}
\hat{A}(E)=\prod_i \frac{x_i/2}{\sinh(x_i/2)}\in H^*(X,\mathbb{Q})
\end{equation*}

\begin{remark}
Let $E\rightarrow X$ be a real vector bundle with a given orientation and a spin$^c$ structure. Then we see that $E$ is orientable, which means that its frame bundle $F(E)$ has two connected components. $F(E)$ is a fiber bundle over $X$ with projection $\pi$. Let $F^+(E)$ be the component that corresponds to the orientation of $E$. Notice the spin$^c$ structure on $E$ is given by a cohomology class $u\in H^2(F^+(E),\mathbb{Z})$ such that $i^*_x(u)\neq 0$ for all $x\in X$, where $i^*_x:H^2(F^+(E),\mathbb{Z})\rightarrow H^2(F^+_x(E),\mathbb{Z})\cong \mathbb{Z}_2$ is the pullback of the natural inclusion of fiber $i:F^+_x(E)\hookrightarrow F^+(E)$. Then one can define the \emph{first Chern class} of $E$ as $c_1(E)\in H^2(X,\mathbb{Z})$ given by $\pi^*(c_1(E))=2u$.
\end{remark}

\begin{definition}
The \emph{Todd class} of a spin$^c$ vector bundle $E\rightarrow X$ is defined as
\begin{equation*}
Td(E)=e^{c_1(E)/2}\hat{A}(E)
\end{equation*}
\end{definition}

\section{Topological $K$-theory}
Topological $K$-theory is a version of $K$-theory in algebraic topology, due to Michael Atiyah and Friedrich Hirzebruch \cite{atiya-hirzebruch}. Let $X$ be a compact Hausdorff space. Then the $K$-theory group of $X$ is an abelian group generated by stable isomorphism classes of vector bundles over $X$.

\begin{theorem}
There exists an abelian group $K^0(X)$ and a semigroup homomorphism $\phi: Vect(X)\rightarrow K^0(X)$, with the following universal property: for each abelian group $G$ and semigroup homomorphism $\psi: Vect(X)\rightarrow G$, there exists a unique group homomorphism $f: K^0(X)\rightarrow G$ such that the following diagram commutes
\begin{equation*}
\begindc{\commdiag}[50]
\obj(0,1)[01]{$Vect(X)$}
\obj(2,1)[21]{$K^0(X)$}
\obj(1,0)[10]{$G$}
\mor{01}{21}{$\phi$}
\mor{01}{10}{$\psi$}
\mor{21}{10}{$f$}
\enddc
\end{equation*}
$K^0(X)$ is unique up to isomorphisms that commute with $\phi$.
\end{theorem}

\begin{proof}
To prove the existence, we will construct $K^0(X)$ explicitly. Define
\begin{itemize}
\item $F(X):=$  the free abelian group generated by elements of $Vect(X)$
\item $E(X) :=$ the subgroup of $F(X)$ generated by $\{E+E'-(E\oplus E')|E,E'\in Vect(X)\}$
\end{itemize}
Then we can construct $K^0(X):= F(X)/E(X)$. Accordingly, we define $\phi: Vect(X)\rightarrow K^0(X)$ as the projection $\phi(E)=[E]$.

To prove the universal property, note that $f$ is already uniquely defined by the explicit construction above, namely $f([E])=\psi(E)$. It is also well defined since $\psi$ is a homomorphism, i.e. $\psi(E)+\psi(E')=\psi(E\oplus E')$ for all $ E,E'\in Vect(X)$.
\end{proof}

\begin{definition}
The abelian group $K^0(X)$ is called the \emph{$K$-theory group} of the space $X$.
\end{definition}

Note that the elements of $K^0(X)$ are not just isomorphism classes of vector bundles over $X$. One can show that any two elements $[E],[F]\in K^0(X)$ are the same if and only if they are \emph{stably isomorphic}, i.e. there exists  a trivial bundle $\varepsilon^n\rightarrow X$ such that $E\oplus \varepsilon^n = F\oplus \varepsilon^n$. However, $K^0(X)$ is not a group of stable isomorphism classes of vector bundles under direct sum, because such a group would have no inverses. For example, $E\oplus F$ is stably isomorphic to $\varepsilon^0$ if and only if $E\oplus F\oplus \varepsilon^n\cong \varepsilon^n$ for some $n$, which means both $E$ and $F$ have rank $0$.

One can show that the elements of $K^0(X)$ are of the form $E-F$, a formal difference between ``a pair of bundles''. To see this, let's define $\Delta: Vect(X)\rightarrow Vect(X)\times Vect(X)$ to be the diagonal map, which is a semigroup homomorphism. Then let's denote the quotient as
\begin{equation*}
\mathcal{K}(X):= Vect(X)\times Vect(X)/\Delta(Vect(X))
\end{equation*}
which is in fact a group since one can induce inverses in it by interchanging factors in $Vect(X)\times Vect(X)$.

\begin{theorem}
$\mathcal{K}(X)\cong K^0(X)$.
\end{theorem}

\begin{proof}
Let's define $\phi: Vect(X)\rightarrow \mathcal{K}(X)$ as $\phi(E)= [(E,\varepsilon^0)]$. Let $G$ be any abelian group, with a semigroup homomorphism $\psi: Vect(X)\rightarrow G$. Let $\mathcal{K}(G)$ denote the quotient of $G\times G$ via the diagonal map, with the corresponding group homomorphism $\phi_G$ analogous to $\phi$. Then the induced group homomorphism $K_{\phi}: \mathcal{K}(X)\rightarrow \mathcal{K}(G)$ makes the following diagram commute:
\begin{equation*}
\begin{CD}
Vect(X) @>\psi>> G\\
@VV\phi V @VV\phi_G V\\
\mathcal{K}(X) @>K_{\phi}>> \mathcal{K}(G)
\end{CD}
\end{equation*}
Note $\phi_G$ is in fact an isomorphism because $G$ is an abelian group. Hence $\mathcal{K}(G)$ satisfies the universal property in Theorem 2.3.1, with $f=\phi_G^{-1}\circ K_{\phi}$. The result then follows from the uniqueness.
\end{proof}

Formally, we can write $[(E,F)]\in \mathcal{K}(X)$ as $E-F$. Then $\mathcal{K}(X)$ is an abelian group under the direct sum. Therefore, the identity in $K^0(X)$ is of the form $E-E$ for any $E$, while the inverse of $E-F$ is simply $F-E$. Then $K^0(X)$ can be alternatively defined as the group of elements of the form $E-F$, such that $E-F=(E\oplus G)-(F\oplus G)$ for any vector bundle $G\rightarrow X$. That is to say, $E-F=E'-F'$ in $K^0(X)$ if and only if $E\oplus F'$ is stably isomorphic to $E'\oplus F$. This means every element of $K^0(X)$ can be put in the form $E-\varepsilon^n$.

\begin{remark}
Recall that one can define pullback bundles, which provide a relation between vector bundles over different base spaces $X$ and $Y$. It is straightforward to see that, for $f:X\rightarrow Y$, pullback induces a group homomorphism $f^*:K^0(Y)\rightarrow K^0(X)$. This way, $K$-theory becomes a contravariant functor from the category of topological spaces to the category of abelian groups.

In fact one can do more. The tensor product operation on vector bundles provides a natural product operation on $K^0(X)$ as follows: for any two elements $E_1-E_2,F_1-F_2\in K^0(X)$, one can define
\begin{equation*}
(E_1-E_2)(F_1-F_2)=E_1\otimes F_1-E_1\otimes F_2-E_2\otimes F_1+E_2\otimes F_2
\end{equation*}
which makes $K^0(X)$ into a commutative ring with identity $\varepsilon^1$. Then $K$-theory is in fact a contravariant functor from the category of topological spaces to the category of commutative rings.
\end{remark}

\begin{example}
Let $X=$\{point\}. Then every vector bundle over $X$ is trivial, hence characterized by its rank. It then follows that $K^0($point$)\cong \mathbb{Z}$.
\end{example}

\begin{definition}
Specify a point $x_0\in X$ by the inclusion map $i:x_0\rightarrow X$. Then we define the \emph{reduced $K$-theory group} of $X$ as
\begin{equation*}
\tilde{K}^0(X)= ker\left(i^*: K^0(X)\rightarrow K^0(x_0)\cong \mathbb{Z}\right).
\end{equation*}
\end{definition}

An element in the reduced $K$-theory group $\tilde{K}^0(X)$ is of the form $E-F\in K^0(X)$ where $E$ and $F$ have the same rank. Equivalently, it can be written as $E-\varepsilon^n$ where $E$ has rank $n$. Note in $\tilde{K}^0(X)$, we have $E-\varepsilon^n=F-\varepsilon^m$ if and only if $[E\oplus\varepsilon^m]=[F\oplus\varepsilon^n]$ in $K^0(X)$, which is true if and only if there exists a $ k\in \mathbb{Z}^+$ such that $E\oplus\varepsilon^m\oplus\varepsilon^k=F\oplus\varepsilon^n\oplus\varepsilon^k$ in $Vect(X)$, i.e. $E\oplus\varepsilon^{m+k}=F\oplus\varepsilon^{n+k}$. Two vector bundles $E$ and $F$ with this property are called \emph{stably equivalent} to each other. Therefore, $\tilde{K}^0(X)$ consists of the stable equivalent classes of vector bundles over $X$. Clearly $\tilde{K}^0(X)$ is an ideal of $K^0(X)$, with the following short exact sequence
\begin{equation*}
0\rightarrow \tilde{K}^0(X)\rightarrow K^0(X)\rightarrow \mathbb{Z}\rightarrow 0
\end{equation*}
which also splits.

For any CW complex $X$ with $dim X\leq 2n$, one can prove
\begin{equation*}
\tilde{K}^0(X)\simeq [X, G(n,2n)]
\end{equation*}
where $[X, G(n,2n)]$ is the set of homotopy classes of maps from $X$ to the Grassmannian $G(n,2n)$, and $\simeq$ is a bijective relation between sets. Therefore, one can view $K$-theory as a generalized cohomology theory. As such, we can define the following:

\begin{definition}
Let $Y$ be a closed subspace of $X$. Then we define the \emph{relative $K$-theory}
\begin{equation*}
K^0(X,Y):= \tilde{K}^0(X/Y)
\end{equation*}
\end{definition}

\begin{definition}
Let $X$ be a locally compact Hausdorff space. The \emph{$K$-theory with compact support} of $X$ is
\begin{equation*}
K^0_c(X):= K^0(X^+, x_0)
\end{equation*}
where $X^+$ is the one point compactification of $X$, with the added point being $x_0$.
\end{definition}

An element of $K^0(X,Y)$ is of the form $[E,F;t]$, where $E$ and $F$ are vector bundles over $X$, with the property that $t: E|_Y\rightarrow F|_Y$ is an isomorphism. Later we will see that $t$ is naturally related to the tachyon condensation of D-branes.

Since $K$-theory is a generalized cohomology theory, one would expect there are higher degree $K$-theory groups $K^n(X)$, as suggested by our notation $K^0(X)$. Indeed, in the case of complex vector bundles, there are essentially two different $K$-theory groups, namely $K^{n}(X)=K^0(X)$ if $n$ is even, while $K^{n}(X)=K^1(X)$ if $n$ is odd.\footnote{In the case of real vector bundles, there are eight $K$-theory groups, denoted by $KO^n(X)$, $n=0,1,...,7$.} This phenomenon is called \emph{Bott periodicity} \cite{bott-peorid}.  Therefore, we can define the \emph{total $K$-theory group} of $X$ as the graded ring
\begin{equation*}
K^*(X):= K^0(X)\oplus K^1(X)
\end{equation*}
which will often be called the $K$-theory group of $X$, too.

\begin{example}
If $X$ is a point, then $K^1(X)=0$. Thus $K^*(X)= K^0(X)\cong \mathbb{Z}$.
\end{example}

\begin{remark}
Recall that the Chern character of vector bundles over $X$ is the map from $Vect(X)$ to $H^*(X,\mathbb{Q})$. In fact, it is a natural transformation between two contravariant functors: $K^*(-)$ and $H^*(-,\mathbb{Q})$. Explicitly, it provides a map $ch^{\bullet}: K^0(X)\rightarrow H^*(X,\mathbb{Q})$, via $ch^{\bullet}(E-F)=ch^{\bullet}(E)-ch^{\bullet}(F)$, which is in fact a ring homomorphism. We will call it the \emph{cohomological Chern character}. Later we will see that it is related to the charges of D-branes.

Note that $ch^{\bullet}$ maps $K^0(X)$ to $ H^{even}(X,\mathbb{Q})$, and maps $K^1(X)$ to $ H^{odd}(X,\mathbb{Q})$, where
\begin{equation*}
\begin{split}
& H^{even}(X,\mathbb{Q})=\bigoplus_{n\geq 0} H^{2n}(X,\mathbb{Q})\\
& H^{odd}(X,\mathbb{Q})=\bigoplus_{n\geq 0} H^{2n+1}(X,\mathbb{Q}).
\end{split}
\end{equation*}
\end{remark}

\section{D-brane Charges}
There is not an exact definition of D-branes, as the exact formulation of string theory itself is still unknown. However, we do know that D-branes are certain dynamical extended objects in string theory that open strings can ``end on''. In many cases, we can model a D-brane as a subspace of spacetime (called the \emph{worldvolume} of this D-brane), with a vector bundle over it (called the \emph{gauge bundle}). We say that, in this case, a D-brane \emph{wraps} the subspace. This interpretation will suffice for our purpose in this chapter in discussing the relationship between D-branes and $K$-theory.

More accurately, one basic property of D-branes is that they are charged under the Ramond-Ramond fields from the quantum spectrum of open strings. Originally, D-brane charges were thought to be cohomology classes of spacetime, via some anomaly computation \cite{harvey-moore-khomology, freed-witten-anomaly}. In \cite{witten-ktheory}, Witten pointed out that the \emph{charges of various D-brane configurations should be interpreted as elements in the $K$-theory group of spacetime}; the previous cohomology classes that were regarded as D-brane charges are actually images of these $K$-theory elements under the ``modified Chern homomorphism'' \cite{harvey-moore-khomology}.

We will mostly focus on D-branes from type IIA and type IIB string theories. In type IIA string theory, D-branes have odd dimensional worldvolumes, while type IIB D-branes have even dimensional worldvolumes. The gauge bundles in both cases are complex vector bundles with structure group $U(n)$, where $n$ is the rank of the vector bundle, which is also the number of D-branes that wrap this subspace. Traditionally, people denote a D-brane with $p+1$ dimensional worldvolume as a Dp-brane.

The relationship between these D-branes and $K$-theory is most directly established via the Thom isomorphism in $K$-theory, so let's take some time to discuss this theorem (we will be mostly following \cite{thom}). Let $\pi: E\rightarrow M$ be a real or complex vector bundle, where $M$ is a compact space. Then one can prove that $K^*_c(E)$ is in fact a $K^*(M)$-module.

\begin{definition}
Let $u\in K^0_c(E)$. If $K^*_c(E)$ is a free $K^*(M)$-module generated by $u$, then $u$ is called a \emph{$K$-theory orientation} of $E$.
\end{definition}

In discussing D-branes, we will focus on spin$^c$ vector bundles, although the argument can be easily applied to other vector bundles. So let's take $\pi: E\rightarrow M$ to be an oriented spin$^c$ vector bundle of rank $2n$, with the associated irreducible spinor bundle denoted as $S(E)=S^+(E)\oplus S^-(E)$.

\begin{proposition}
(\textbf{Thom isomorphism}) Let $\pi: E\rightarrow M$ be an oriented spin$^c$ vector bundle of rank $2n$. There is a natural $K$-theory orientation of $E$, namely
\begin{equation*}
s(E)=[\pi^*S^+(E),\pi^*S^-(E);f]\in K^0_c(E)
\end{equation*}
where $f$ is defined by the Clifford multiplication, namely $f_e(\alpha)=e\cdot \alpha$, for all $e\in E$. Then the associated map
\begin{equation} \label{thom}
\begin{split}
i_!:K^0(M)&\rightarrow K^0_c(E)\\
a&\mapsto (\pi^*a)\cdot s(E)
\end{split}
\end{equation}
is an isomorphism.
\end{proposition}

Let's now return to D-branes. Let $X$ be the ten dimensional spacetime of string theory, which is a locally compact spin manifold. For simplicity and clarity, let's first focus on D-branes in type IIB string theory. These are submanifolds of $X$ with dimensions 0, 2, 4, 6, 8, and 10. Let's also restrict to the case where the field strength 3-form $H$ of the Neveu-Schwarz 2-form field $B$ is cohomologically trivial\footnote{The case with nontrivial $B$ will involves more complicated versions of $K$-theory, such as the twisted $K$-theory, which will not be discussed here.}. Consider k D$p$-branes wrapping a $p+1$ dimensional compact submanifold M of $X$ with embedding $i:M\hookrightarrow X$. Let $N\rightarrow M$ denote the normal bundle of $M$, defined by the short exact sequence of vector bundles
\begin{equation*}
0\rightarrow TM\rightarrow TX|_M\rightarrow N\rightarrow 0
\end{equation*}
where $TM$ is the tangent bundle of $M$, and $TX|_M = i^*TX$ is the pullback (or restriction) of the tangent bundle of $X$ to $M$. Note that $N$ can be identified with a tubular neighbourhood of $M$ in $X$. Let's assume that $N$ is actually a spin$^{c}$ vector bundle, with the physical justification being explained momentarily.

Recall that there is a gauge bundle $V\rightarrow M$ on $M$, which is a vector bundle with rank $k$ and defines an element $V\in K^0(M)$. Then one can obtain $(V\otimes S^+(N), V\otimes S^-(N)) \in K^0(M)$ as the first step of applying Thom isomorphism to $N$, where $S(N)=S^+(N)\oplus S^-(N)$ is the spinor bundle associated with $N$. Note that the Thom isomorphism gives $K^0(M)\cong K^0_c(N)$. Then together with the excision homomorphism $j_*:K^0_c(N)\rightarrow K^0_c(X)$ induced by the natural inclusion $j: N\hookrightarrow X$, we obtain a homomorphism of rings
\begin{equation} \label{abs}
\Psi = j_*\circ i_!: K^0(M)\rightarrow K^0(X)
\end{equation}
Hence, our D$p$-brane wrapping on $M$ determines an element in $K^0(X)$. Similarly, we see that every IIB D$p$-brane is represented by a K-theory class on the entire spacetime.

We can work out the explicit form of $\Psi(V)$ as follows. First, Thom isomorphism gives
\begin{equation*}
i_!(V)=[\pi^*(V\otimes S^+(N)),\pi^*(V\otimes S^-(N));f]
\end{equation*}
where $f$ is defined by the Clifford multiplication.  Recall Lemma 2.1.6 says that, for any vector bundle over a compact Hausdorff space, there exists a complementary vector bundle on the same space such that the direct sum of these two vector bundles is isomorphic to a trivial bundle. Let's apply this lemma to $\pi^*(V\otimes S^+(N))$, i.e. there exists a vector bundle $W\rightarrow N$ such that $\pi^*(V\otimes S^+(N))\oplus W \cong \varepsilon^n$ for some positive integer $n$. This way, we obtain an element in $K^0_c(N)$ of the form $[\varepsilon^n, \pi^*(V\otimes S^-(N))\oplus W; f\oplus id_W]$. Next, we extend the trivial bundle $\varepsilon^n$ to the entire $X$, which we denote as $E$. Then we can obtain an element $(E, F)\in K^0_c(X)$ by demanding $F$ is isomorphic to $E$ on $X-M$, and $F|_M = V\otimes S^-(N)\oplus W$. Note $E$ and $F$ have the same rank, which means if $X$ is compact, then what we actually obtain is an element in $\tilde{K}^0(X)$.

\begin{remark}
This $K$-theoretic construction has a very natural physical interpretation. First of all, anomalies prevent D-branes from wrapping submanifolds whose normal bundles are not spin$^c$ \cite{witten-anomaly,freed-witten-anomaly}, justifying our assumption. The vector bundle $E$ is associated to $n$ anti-D9-branes wrapping the spacetime $X$. Similarly, $F$ represents $n$ D9-branes wrapping $X$. The fact that $E$ and $F$ are isomorphic on $X-M$ is physically realized as the tachyon condensation, through which the D9-branes and anti-D9-branes annihilate each other on $X-M$, leaving $k$ D$p$-branes wrapping $M$ \cite{witten-ktheory}. These physical processes are most explicitly captured by the $K$-theoretic description of D-brane charges as above, which is one of the major reasons of using $K$-theory, rather than cohomology, to classify D-brane charges.

There is another major advantage of this $K$-theory construction. The mathematical fact that there are only two $K$-theory groups is reflected in the physical fact that there  are two different type II string theories: type IIA and type IIB. Physically, type IIA string theory is related to type IIB string theory via T-duality. Therefore, the IIA branes ought to be described by some $K$-theoretic quantities. It turns out that they are classified by $K^1(X)$ \cite{witten-ktheory}. Therefore, the physical fact that there are only two kinds of type II string theory can be seen from the mathematical Bott periodicity \cite{witten-ktheory}. Hence, all D-branes in type II string theory are classified by the total $K$-theory group $K^*(X)$.\footnote{D-branes charges in type I string theory are classified by the real $K$-theory, $KO^*(X)$ \cite{witten-ktheory}.}
\end{remark}

Although this $K$-theoretic classification of D-branes charges is very powerful, there are certain drawbacks. First of all, it is hard to do the reverse engineering, i.e. given an abstract element in $K^*(X)$ representing the charge of a D-brane configuration, it is difficult to obtain a direct geometric picture of D-branes wrapping some submanifold. It would be more powerful to have a classification of D-branes that captures the geometric information of D-branes directly. Second, the $K$-theoretic interpretation treats pairs of D9-anti-D9 branes quite differently from those lower dimensional D-branes.

In the next chapter, we will study $K$-homology, the homological dual of $K$-theory. We will see that topological $K$-homology provides nice solutions to the above issues and subtleties, with further implications for the physics of D-branes.

\pagebreak

\chapter{$K$-homology and D-branes}

\section{Topological $K$-homology}
$K$-homology is the homological dual of $K$-theory. The topological version of $K$-homology, dual to the topological $K$-theory discussed in the last chapter, was formulated by Baum and Douglas \cite{baum-douglas}. As observed there, the isomorphism between topological $K$-homology and analytical $K$-homology provides a unified framework for proving various index theorems. For us, topological $K$-homology gives a direct geometric way of classifying D-branes.

The topological $K$-homology group of a topological space contains the $K$-theoretic version of cycles, which are defined as the following
\begin{definition}
A \emph{$K$-cycle} for a topological space $X$ is a triple $[M,E,\phi]$, where
\begin{itemize}
\item $M$ is a compact spin$^c$ manifold without boundary,
\item $E\stackrel{\pi}{\rightarrow} M$ is a complex vector bundle,
\item $\phi:M\rightarrow X$ is a continuous map.
\end{itemize}
\end{definition}

\begin{definition}
Two $K$-cycles $[M,E,\phi]$ and $[M',E',\phi']$ are called \emph{isomorphic}, denoted $[M,E,\phi]\cong [M',E',\phi']$, if there exists a diffeomorphism $h: M\rightarrow M'$ such that
\begin{itemize}
\item $h$ preserves the spin$^c$ structure,
\item The pullback bundle $h^*(E')\cong E$, and
\item The following diagram commutes
\begin{equation*}
\begindc{\commdiag}[50]
\obj(0,1)[01]{$M$}
\obj(2,1)[21]{$M'$}
\obj(1,0)[10]{$X$}
\mor{01}{21}{$h$}
\mor{01}{10}{$\phi$}
\mor{21}{10}{$\phi'$}
\enddc
\end{equation*}
\end{itemize}
\end{definition}

The set of all isomorphism classes of $K$-cycles on $X$ is denoted as $\Gamma(X)$. Note that disjoint union is a well defined operation on $\Gamma(X)$.

\begin{definition}
Given a $K$-cycle $[M,E,\phi]$ for a space $X$, together with a real smooth spin$^c$ vector bundle $f:H\rightarrow M$ of even rank, we can define the \emph{clutching construction} as the following procedure:
\begin{itemize}
\item Let $L=M\times \mathbb{R}$ be the trivial real line bundle. Then define a sphere bundle
      \begin{equation*}
      \rho:\hat{M}:= S(H\oplus L)\rightarrow M
      \end{equation*}
      i.e. the unit sphere bundle of $H\oplus L$.
\item Let $n$ be the rank of $H$. Because $n$ is even, we have the associated irreducible spinor bundle $S=S^+\oplus S^-$. Then we define complex vector bundles $S^0=f^*S^+$ and $S^1=f^*S^-$ over $H$. Note that the paring $H\times S^+\rightarrow S^-$, defined by the Clifford multiplication, induces a vector bundle map
    \begin{equation*}
    \sigma:S^0\rightarrow S^1
    \end{equation*}
    which is an isomorphism off the zero section of $H$.
\item Let $B(H)$ be the unit ball bundle of $H$. Then one can write
    \begin{equation*}
    \hat{M}=B(H)\cup_{S(H)}B(H)
    \end{equation*}
    where the two copies of $B(H)$ are glued together by $id_{S(H)}$ on their boundaries. Note the complex vector bundles $S^0$ and $S^1$ over $H$ induce complex vector bundles $S^0$ and $S^1$, each of which is over the one copy of $B(H)$. Then we define a complex vector bundle
    \begin{equation*}
    \hat{H}:= S^0\cup_{\sigma} S^1 \rightarrow \hat{M}
    \end{equation*}
    where $S^0$ and $S^1$ are attached together by $\sigma$ on $S(H)$.
\end{itemize}
The clutching construction generates another $K$-cycle $[\hat{M},\hat{H}\otimes \rho^*(E),\phi\circ\rho]$ of $X$.
\end{definition}

\begin{definition}
The \emph{topological $K$-homology group} of a topological space $X$ is defined to be $K^t_*(X):=\Gamma(X)/\sim$, where $\sim$ is the equivalence relation generated by the following three equivalences:
\begin{itemize}
\item \emph{Bordism}: $[M_1,E_1,\phi_1]\sim [M_2,E_2,\phi_2]$ if there is a compact spin$^c$ manifold $W$ with boundary, a complex vector bundle $E\rightarrow W$, and a continuous map $\phi:W\rightarrow X$, such that
    \begin{equation*}
    [\partial W,E|_{\partial W},\phi|_{\partial W}]\cong [M_1,E_1,\phi_1]\amalg[-M_2,E_2,\phi_2]
    \end{equation*}
    where $-M_2$ is $M_2$ with the reversed spin$^c$ structure.
\item \emph{Direct sum}: $[M,E_1,\phi]\amalg[M,E_2,\phi]\sim[M,E_1\oplus E_2,\phi]$.
\item \emph{Vector bundle modification}: Given a $K$-cycle $[M,E,\phi]$, together with a real smooth spin$^c$ vector bundle $f:H\rightarrow M$ of even rank, we have
     \begin{equation*}
     [M,E,\phi]\sim [\hat{M},\hat{H}\otimes \rho^*(E),\phi\circ\rho]
     \end{equation*}
     where $\hat{M},\hat{H}$ and  $\rho$ are defined by the clutching construction above.
\end{itemize}
\end{definition}

\begin{remark}
$K^t_*(X)$ is an abelian group, with the group operation given by the disjoint union. Note that $\sim$ preserves the parity of the dimension of $M$. Therefore, $K^t_*(X)$ in fact has grading
\begin{equation*}
K^t_*(X)=K^t_0(X)\oplus K^t_1(X)
\end{equation*}
where $K^t_0(X)$ and $K^t_1(X)$ are subgroups given by all $[M,E,\phi]$ with each connected component of $M$ even or odd dimensional, respectively. Unlike the topological $K$-theory groups $K^0(X)$ and $K^1(X)$, these two topological $K$-homology groups are very much like each other, in the sense that their elements can be directly described in a unified way.
\end{remark}

\begin{proposition}
Let $f: X\rightarrow Y$ be a continuous map. Then there exists a homomorphism $f_*:K^t_*(X)\rightarrow K^t_*(Y)$, given by $f_*[M,E,\phi]=[M,E,f\circ\phi]$.
\end{proposition}
Note this makes $K^t_*(-)$ into a covariant functor from the category of topological spaces to the category of abelian groups.

\begin{definition}
The \emph{homological Chern character} $ch_{\bullet}$ is a natural transformation between two covariant functors
\begin{equation*}
ch_{\bullet}:K^t_*(-)\rightarrow H_*(-,\mathbb{Q})
\end{equation*}
Explicitly, let $[M,E,\phi]$ be a $K$-cycle on a topological space $X$. Let $\phi_*:H_*(M,\mathbb{Q})\rightarrow H_*(X,\mathbb{Q})$ be the pushforward map induced by $\phi$. Then
\begin{equation} \label{homology-chern}
ch_{\bullet}[M,E,\phi]=\phi_*\big(\left(ch^{\bullet}(E)\cup Td(M)\right)\cap[M]\big)
\end{equation}
where $\left(ch^{\bullet}(E)\cup Td(M)\right)\cap[M]$ is the Poincar\'{e} dual of $ch^{\bullet}(E)\cup Td(M)$.
\end{definition}

\section{Analytic $K$-homology}
There is another presentation of $K$-homology called analytic $K$-homology, which is the operator theory version of $K$-homology. One can use Dirac operators on spin$^c$ manifolds to construct an isomorphism between topological $K$-homology and analytic $K$-homology. This isomorphism provides a unified framework for various index theorems.

Before discussing analytic $K$-homology, we will first need some concepts from operator theory. Let $H_0,H_1$ be two separable Hilbert spaces.

\begin{definition}
A bounded operator $T:H_0\rightarrow H_1$ is called \emph{Fredholm} if $T$ has finite dimensional kernel and cokernel. Then we say the \emph{index} of $T$ is
\begin{equation}
Index(T):= dim \ ker(T)- dim\ coker(T)
\end{equation}
\end{definition}

Let $H$ be a separable Hilbert space. We will use the following notation:
\begin{itemize}
\item $L(H)=C^*$-algebra of all bounded operators on $H$.
\item $K(H)=$ closed ideal in $L(H)$ of all compact operators on $H$.
\item $Q(H)=L(H)/K(H)$.
\end{itemize}

Given a topological space $X$, let $C(X)$ denote the $C^*$-algebra of all continuous complex-valued functions on $X$. Let's now consider the set $S(X)$ of 5-tuples $(H_0,\psi_0,H_1,\psi_1,T)$, where
\begin{itemize}
\item $H_0,H_1$ are separable Hilbert spaces.
\item $\psi_i: C(X)\rightarrow L(H_i)$, $i=0,1$ is an unital algebra *-homomorphism.
\item $T:H_0\rightarrow H_1$ is a bounded Fredholm operator, with the property that $T\psi_0(f)-\psi_1(f)T$ is compact for all $f\in C(X)$.
\end{itemize}

\begin{definition}
We define some relations on the set $S(X)$:
\begin{itemize}
\item \emph{Isomorphism}: $(H_0,\psi_0,H_1,\psi_1,T)\cong (H_0',\psi_0',H_1',\psi_1',T')$ if there exist unitary operators $U_0:H_0\rightarrow H_0'$ and $U_1:H_1\rightarrow H_1'$, such that $T'=U_1TU_0^*$ and $\psi_i'(f)=U_i\psi_i(f)U_i^*$ for all $f\in C(X)$.
\item \emph{Triviality}: An element of $S(X)$ is called trivial if it is of the form $(H,\psi,H,\psi,id_H)$.
\item \emph{Direct sum}:
\begin{equation*}
(H_0,\psi_0,H_1,\psi_1,T)\oplus(H_0',\psi_0',H_1',\psi_1',T')=(H_0\oplus H_0',\psi_0\oplus \psi_0',H_1\oplus H_1',\psi_1\oplus \psi_1',T\oplus T')
\end{equation*}
\end{itemize}
\end{definition}

\begin{definition}
The \emph{analytic $K$-homology group} of degree 0 of $X$ is $K^a_0(X):= S(X)/\sim$, where $\sim$ is the equivalence relation generated by the following equivalence relations:
\begin{itemize}
\item Isomorphic objects are equivalent.
\item $(H_0,\psi_0,H_1,\psi_1,T)\sim(H_0\oplus H,\psi_0\oplus \psi,H_1\oplus H,\psi_1\oplus \psi,T\oplus id_H)$.
\item Given two objects of the form $(H,\psi,H,\psi,T)$ and $(H',\psi',H',\psi',T')$, let $V,V'$ be the partial isometry parts in the polar decompositions of $T,T'$. Then $(H,\psi,H,\psi,T)\sim(H',\psi',H',\psi',T')$ if $V-V'$ is compact, and $\psi(f)-\psi'(f)$ is compact for all $f\in C(X)$.
\end{itemize}
\end{definition}
$K^a_0(X)$ is an abelian group under direct sum. The inverse of $(H_0,\psi_0,H_1,\psi_1,T)$ is $(H_1,\psi_1,H_0,\psi_0,T^*)$.

\begin{example}
Consider the n-dimensional sphere $S^n$. Then $K^a_0(S^n)=\mathbb{Z}$ if $n$ is odd, while $K^a_0(S^n)=\mathbb{Z}\oplus\mathbb{Z}$ if $n$ is even.
\end{example}

\begin{remark}
Notice we can define a homomorphism
\begin{equation*}
\begin{split}
&Index: K^a_0(X)\rightarrow \mathbb{Z}\\
\text{by}\ & Index(H_0,\psi_0,H_1,\psi_1,T)=Index(T)
\end{split}
\end{equation*}
which clearly is a surjective map.
\end{remark}

Analogously, the degree 1 $K$-homology group of $X$ is obtained from the set $S'(X)$ of pairs of the form $(H,\tau)$, where $H$ is a separable Hilbert space, and $\tau:C(X)\rightarrow Q(H)$ is an unital algebra *-homomorphism.

\begin{definition}
We define some relations on the set $S'(X)$:
\begin{itemize}
\item \emph{Unitary equivalence}: $(H_0,\tau_0)\cong (H_1,\tau_1)$ are unitarily equivalent if there exists an unitary operator $U:H_0\rightarrow H_1$ such that the following diagram is commutative:
\begin{equation*}
\begindc{\commdiag}[50]
\obj(0,1)[01]{$C(X)$}
\obj(2,1)[21]{$Q(H_0)$}
\obj(1,0)[10]{$Q(H_1)$}
\mor{01}{21}{$\tau_0$}
\mor{01}{10}{$\tau_1$}
\mor{21}{10}{$\tilde{U}$}
\enddc
\end{equation*}
\item \emph{Triviality}: An element $(H,\tau)$ is trivial if there exists a unital algebra *-homomorphism $\eta:C(X)\rightarrow L(H)$ such that the following diagram is commutative:
\begin{equation*}
\begindc{\commdiag}[50]
\obj(0,1)[01]{$C(X)$}
\obj(2,1)[21]{$L(H)$}
\obj(1,0)[10]{$Q(H)$}
\mor{01}{21}{$\eta$}
\mor{01}{10}{$\tau$}
\mor{21}{10}{$\pi$}
\enddc
\end{equation*}
where $\pi$ is the natural projection.
\item \emph{Direct sum}:
\begin{equation*}
(H_0,\tau_0)\oplus(H_1,\tau_1)=(H_0\oplus H_1,\tau_0\oplus \tau_1)
\end{equation*}
\end{itemize}
\end{definition}

\begin{definition}
The \emph{degree 1 analytic $K$-homology group} of a topological space $X$ is $K^a_1(X):= S'(X)/\sim$, where $\sim$ is the equivalence relation given by: $(H_0,\tau_0)\sim(H_1,\tau_1)$ if there exit trivial elements $(H_0',\tau_0')$ and $(H_1',\tau_1')$ such that $(H_0,\tau_0)\oplus(H_0',\tau_0')$ is unitarily equivalent to $(H_1,\tau_1)\oplus(H_1',\tau_1')$
\end{definition}
$K^a_1(X)$ is an abelian group under direct sum. Thus we can define the total analytic $K$-homology group as
\begin{equation*}
K^a_*(X)=K^a_0(X)\oplus K^a_1(X)
\end{equation*}
which is a $\mathbb{Z}_2$ graded abelian group under direct sum.

\begin{remark}
Let $X$ be a closed spin$^c$ smooth manifold. Let $E_0,E_1$ be two complex vector bundles over $X$. Let $C(E_i)$ be the vector space of smooth sections of $E_i$, $i=0,1$. One can then construct Hilbert spaces $L^2(E_i)$ from $C(E_i)$ using any Hermitian metric on $E_i$. Consider an elliptic degree zero pseudo-differential operator $D:C(E_0)\rightarrow C(E_1)$, which can be viewed as an operator $D:L^2(E_0)\rightarrow L^2(E_1)$. Then one can show that every element of $K^a_0(X)$ can be obtained from some $D$ of this type \cite{baum-douglas}. Therefore, one can write the element of $K^a_0(X)$ corresponding to a given operator $D$ as $[D]$. Similarly, one can show that all the elements of $K^a_1(X)$ can be obtained using self-adjoint elliptic operators from $C(E)$ to itself, where $E$ is a complex vector bundle on $X$.
\end{remark}

In the next section, we will use an elliptic degree zero pseudo-differential operator constructed from the Dirac operator, which we define now. Let $X$ be a Riemannian manifold with a spin structure. Let $S$ be the spinor bundle with respect to this spin structure. Naturally, the Riemannian structure on $X$ leads to a Riemannian structure on $S$, which in turn determines a covariant derivative $\nabla$ on $S$.
\begin{definition}
The \emph{Dirac operator} on $X$ is defined to be a first order differential operator $D:\Gamma(S)\rightarrow\Gamma(S)$, given by
\begin{equation*}
D\sigma := \sum_{i=1}^n e_i\cdot \nabla_{e_i}\sigma
\end{equation*}
at $x\in X$ for any $\sigma\in \Gamma(S)$, where $e_i$, $i=1,2,...,n$ is an orthonormal basis of $T_x(X)$, and $\cdot$ denotes the Clifford multiplication.
\end{definition}

\section{D-branes Again}

Let's now consider D-branes in type II string theory. Let $X$ denote the ten-dimensional spacetime manifold of string theory. The formulation of the topological $K$-homology group $K^t_*(X)$ using $K$-cycles is clearly geometric in nature. Since the $K$-theory group $K^*(X)$ classifies the charges of D-branes in type II string theory, it is natural to ask what is the relationship between the $K$-homology group $K^t_*(X)$ and type II D-branes. In this section, we will argue that the elements, represented by $K$-cycles, in the \emph{$K$-homology group $K^t_*(X)$ classify D-branes themselves}.

First, let us start with a $K$-cycle $[M,E,\phi]$, where $\phi:M\rightarrow X$ is an embedding. Notice that $M$ is a spin$^c$ manifold by definition. This coincides with the known fact from section 2.3, that D-branes can only wrap submanifolds with spin$^c$ structures, following from two independent arguments using $K$-theory and anomaly computation. In addition, we naturally have a vector bundle $E$ over $M$. Based on these observations, we propose that this $K$-cycle represents a D-brane configuration that wraps $M$, with gauge bundle $E$.

What can we say about the equivalence relations used to define the $K$-homology group from $K$-cycles? Recall that in type IIB string theory, the physical process of tachyon condensation can be mathematically interpreted as the equivalence relation used to define the $K$-theory group of spacetime. More precisely, the configuration of D9-branes and anti-D9-branes wrapping the entire spacetime, represented by a $K$-theory element $(V,W)\in K^0(X)$, can physically reduce to a lower dimensional D-brane configuration on a submanifold $M$ (with gauge bundle $E$ over $M$) via tachyon condensation. We will now introduce a more direct $K$-homological interpretation of tachyon condensation, in the following steps:
\begin{itemize}
\item Let's start with a type IIB D-brane configuration represented by a $K$-cycle $[M,E,\phi]$, where $\phi:M\rightarrow X$ is an embedding. In section 2.4 we explained that we can obtain an element $(\mathcal{E},\mathcal{F})\in K^0_c(N)$, where $N$ is a tubular open neighborhood of $M$ in $X$, identified with the normal bundle of $M$ in $X$, $\mathcal{E}=\pi^*(E\otimes S^+(N))$, and $\mathcal{F}=\pi^*(E\otimes S^-(N))$.

\item Recall that we can extend $(\mathcal{E},\mathcal{F})$ to get an element $(V,W)\in K^0(X)$, which represents the charge of this D-brane configuration. Here $V = \mathcal{E}\oplus G$ is a trivial bundle over $X$, while $W|_{N} = \mathcal{F}\oplus G$, with an isomorphism $f: V\rightarrow W$ on $X-\phi(M)$.

\item Let's define $X_1=X_2=X\times [0,1/2]$. We put opposite orientations on these two copies. Note we can choose to extend $V$ over $X_1$, and $W$ over $X_2$. Then we glue $X_1$ and $X_2$ together, via $(x,1/2)\sim (x,1/2)$ for all $x\in X-N$. The two vector bundles $V\rightarrow X_1$ and $W\rightarrow X_2$ are in turn glued together as well using the above isomorphism $f$.

\item Let $\hat{N}$ denote the double of $N$, obtained via the above glueing process on the boundary of $N$. We end up with a bordism equivalence between two $K$-cycles on X, namely:
     \begin{equation*}
     [X\amalg X, V\amalg W, id_X]\sim [\hat{N}, (\mathcal{E}\oplus G)\cup_{f\oplus id_G}(\mathcal{F}\oplus G), i\cup i]
     \end{equation*}
     where $(\mathcal{E}\oplus G)\cup_{f\oplus id_G}(\mathcal{F}\oplus G)$ is made from $V$ on one copy of $N$ and $W$ on the other copy of $N$ glued together by $f$ on the boundary of $N$. The map $i\cup i:\hat{N}\rightarrow X$ arises from the inclusion $i:N\hookrightarrow X$.

\item Note under the direct sum equivalence, we have
      \begin{equation*}
      [\hat{N}, (\mathcal{E}\oplus G)\cup_{f\oplus id_G}(\mathcal{F}\oplus G), i\cup i] \sim [\hat{N}, \mathcal{E}\cup_{f}\mathcal{F}, i\cup i]\amalg[\hat{N},  G\cup_{id_G} G, i\cup i]
      \end{equation*}

\item Observe that this map $i\cup i:\hat{N}\rightarrow X$ is homotopic to the map $g:\hat{N}\rightarrow X$ defined by $g = \phi\circ p$, where $p:\hat{N}\rightarrow M$ is the natural projection. Therefore, we have an equivalences
     \begin{equation*}
     \begin{split}
     [\hat{N}, \mathcal{E}\cup_f\mathcal{F}, i\cup i] &\sim [\hat{N}, \mathcal{E}\cup_f\mathcal{F}, g]\\
     [\hat{N},  G\cup_{id_G} G, i\cup i] &\sim [\hat{N},  G\cup_{id_G} G, g]
     \end{split}
     \end{equation*}

\item Let $B$ be the unit ball bundle of $N\oplus \varepsilon$ where $\varepsilon = M\times \mathbb{R}$. Let $q: N\oplus \varepsilon\rightarrow N$ be the projection onto the first summand. Let's denote the restriction of $q$ to $B$ as $q_B$. Also note that the map $g$ extends to a map $g_B:B\rightarrow X$. Therefore, we have a bordism equivalence
    \begin{equation*}
    [\hat{N},  G\cup_{id_G} G, g]\sim 0
    \end{equation*}
    given by $[B, q_B^*G, g_B]$. Therefore, in $K^t_0(X)$ we have
    \begin{equation*}
    [\hat{N}, (\mathcal{E}\oplus G)\cup_{f\oplus id_G}(\mathcal{F}\oplus G), i\cup i] \sim [\hat{N}, \mathcal{E}\cup_{f}\mathcal{F}, g]
    \end{equation*}

\item The above construction defines a vector bundle modification equivalence between two $K$-cycles:
     \begin{equation*}
     [\hat{N}, \mathcal{E}\cup_f\mathcal{F}, g]\sim [M,E,\phi]
     \end{equation*}
     On the other hand, note that under the direct sum equivalence, we have
     \begin{equation*}
     [X\amalg X, V\amalg W, id_X]\sim [X, V- W, id_X]
     \end{equation*}
     where we have $V- W$ because we put opposite orientations on the two copies of $X$. Finally, based on the above equivalences we obtain
     \begin{equation*}
     [M,E,\phi]\sim [X, V- W, id_X].
     \end{equation*}
\end{itemize}
Therefore, our D-brane configuration, represented by the $K$-cycle $[M,E,\phi]$, is $K$-homologically equivalent to the $K$-cycle $[X, V- W, id_X]$, representing the D9-branes and anti-D9-branes that reduce to our D-brane configuration on $\phi(M)$ via tachyon condensation. Here $V$ and $W$ are precisely the gauge bundles on the D9-branes and anti-D9-branes. In another words, the physical process of tachyon condensation can be mathematically interpreted as the equivalence relation that defines the $K$-homology group of spacetime, analogous to the $K$-theory case.

\begin{remark}
In the above construction, we made the identification of the normal bundle $N$ of $M$ as the tubular neighborhood of $M$, and glued two copies of $N$ together on their boundaries to get $\hat{N}$. There is a subtlety, however, in the application of the Thom isomorphism to $N$. Recall the Thom isomorphism involves the Clifford multiplication on $N$, which is a vector bundle. As one moves towards the infinity of the fibers, which are vector spaces, the norm of Clifford multiplication goes to infinity. Therefore, when we applied various constructions on the boundary of the tubular neighborhood, we were actually identifying the tubular neighborhood as the \emph{unit ball bundle of $N$}. Of course $N$ can be identified with its unit ball bundle, and, under this identification, the $K$-theory class defined by Clifford multiplication on the vector bundle agrees with the $K$-theory class defined by Clifford multiplication on the ball bundle, so our construction is well defined.
\end{remark}

Based on the above evidence, we now propose the following: \\
\\
\fbox{\parbox{\textwidth}{
 Let $[M,E,\phi]$ be an arbitrary $K$-cycle of the spacetime $X$. We propose that $[M,E,\phi]$ represents a D-brane configuration wrapping $\phi(M)\subseteq X$.}}
\\

Before we proceed to the main argument, there are some relevant concepts and constructions that need to be introduced first. We will mostly follow the discussion in \cite{atiyah}.

Let $A$ denote the space of Fredholm operators on a Hilbert space $H$. Then $\text{Index}:A\rightarrow \mathbb{Z}$ is a continuous map. Consider a topological space $X$ with a continuous map $T:X\rightarrow A$. This data generates a family of Fredholm operators $T_x$ parameterized by $x\in X$. Suppose that the dimension of ker$(T)_x$ is independent of $x$ for any $x\in X$. Then one can construct two vector bundles over $X$: ker$(T)$ and coker$(T)$.

\begin{proposition}
Let $[X,A]$ denote the set of homotopy classes of maps from $X$ to $A$. The map index$: [X,A]\rightarrow K^0(X)$, defined by index$(T)=(\text{ker}(T),\text{coker}(T))$, is bijective.
\end{proposition}

\begin{remark}
Note that the dimension of ker$(T)$ may not be a constant on $X$. However, this will not raise any serious issue. Briefly, for any continuous map from $X$ to $A$, one can construct an element of $K^0(X)$ without the assumption that the dimension of ker$(T)$ is independent of $x\in X$ \cite{atiyah-ktheory}.
\end{remark}

Let $X$ and $Y$ be two smooth manifolds. We will construct a group homomorphism from $K^0(X\times Y)$ to $K^0(X)$. Let $E$ be a vector bundle over $Y$. Then from Lemma \ref{complementary-bundle} we see there exists another vector bundle $F$ on $Y$ such that $E\oplus F\cong Y\times \mathbb{C}^n$. Equivalently, there is a continuous map $T:Y\rightarrow \text{Proj}(\mathbb{C}^n)$ such that $F$ is isomorphic to the vector bundle of kernels of $T$, where Proj$(\mathbb{C}^n)$ is the space of projections of $\mathbb{C}^n$.

Denote the standard basis of $\mathbb{C}^n$ as $e_1,...,e_n$. Then one can write the map $T$ explicitly as
\begin{equation*}
T(y)e_i=\sum_j T^{ij}(y)e_j
\end{equation*}
for any $y\in Y$. Notice the $T^{ij}$ are elements of $C(Y)$, the space of continuous functions on $Y$. Let $H_1$ and $H_2$ be Hilbert space modules for $C(Y)$. Let $H_a^n=H_a\otimes \mathbb{C}^n$, $a=1,2$. Then we can define bounded operators $T_a$ on $H_a^n$ by
\begin{equation*}
T_a(\alpha\otimes e_i)=\sum_j T^{ij}\alpha\otimes e_j
\end{equation*}
for any $\alpha\in H_a$, where we used the fact that $H_a$ are $C(Y)$-modules.

Let $P:H_1\rightarrow H_2$ be an elliptic operator on $Y$. We define a new operator $Q=T_2\tilde{P}T_1$, where $\tilde{P}:H_1^n\rightarrow H_2^n$ is simply $P\otimes id_{\mathbb{C}^n}$. One can prove that $Q$ is in fact an elliptic operator on $Y$ \cite{atiyah}. Performing this construction at each point $x\in X$, one obtains a family of elliptic operators $Q(x)$. Then the above proposition shows there is a homomorphism
\begin{equation} \label{x*y-x}
\Pi:K^0(X\times Y)\rightarrow K^0(X)
\end{equation}

Let's now get back to discussing D-branes in type II string theory. Recall that physically one needs two pieces of data to define a D-brane configuration: a subspace of $X$ that the D-brane configuration wraps, and a $K$-theory element representing the charge of this D-brane configuration. In order to give a possible physical meaning to an arbitrary $K$-cycle, let's construct a map from $K$-cycles to the $K$-theory group, which classifies D-brane charges.

Consider the type IIB case first. Let $[M,V,\phi]$ be any $K$-cycle. Because $M$ is a smooth manifold, it can be embedded into $\mathbb{R}^n$ for some positive integer $n$. The one-point compactification of $\mathbb{R}^n$ is the $n$-dimensional sphere $S^n$. Therefore there exists an embedding $i:M\hookrightarrow S^n$. Hence one can construct an embedding $f:M\hookrightarrow X\times S^n$, by $f(p)=(\phi(p),i(p))$ for any $p\in M$. We choose $n$ such that $M$ has even codimension under this embedding, which means $n$ must be an even number.

Notice that the normal bundle of $M$ in $X\times S^n$ is a spin$^c$ vector bundle because $M$ is a spin$^c$ manifold. Therefore we can apply the homomorphism (\ref{abs}) to this normal bundle to obtain an element $(\mathcal{E},\mathcal{F})\in K^0(X\times S^n)$, using the vector bundle $V$ on $M$. Let $D$ be the Dirac operator on $S^n$. One can obtain an elliptic order zero pseudo-differential operator from $D$, which will simply be denoted as $D$. From $(\mathcal{E},\mathcal{F})$ and $D$ one obtains an element $(E,F)\in K^0(X)$, by the map (\ref{x*y-x}). By construction, $E$ and $F$ are isomorphic to each other outside an arbitrarily small neighborhood of $\phi(M)$. Therefore, the $K$-cycle $[M,V,\phi]$ represents a IIB D-brane configuration wrapping $\phi(M)$, with charge given by the $K$-theory element $(E,F)\in K^0(X)$.

Analogously, in the type IIA case one can find a map between the $K$-cycles that define $K^t_1(X)$ and the elements in $K^1(X)$, by choosing $n$ so that $M$ has even codimension in $X\times S^n$. In this case, $n$ must be odd, so the Dirac operator on $S^n$ is a self-adjoint operator. One can then use $KK$-theory \cite{blackadar} to argue that this data leads to elements in $K^1(X)$.

\begin{remark}
At the level of $K$-homology, one can prove the $K$-theoretic version of Poincar\'{e} duality, which, in our case, states that $K^*(X)\cong K^t_*(X)$ \cite{reis-szabo-khomology}. Therefore, the $K$-theory group and the $K$-homology group contain the same information as abstract groups. However, they have different (in fact dual) formulations. It is the $K$-cycle formulation of $K$-homology that contains direct geometric information about D-branes.
\end{remark}

\begin{remark}
In general, for a particular $K$-cycle $[M,E,\phi]$ we don't need $\phi(M)$ to be a smooth manifold, because $\phi$ is only required to be continuous. Therefore, our mathematical construction suggests that it is meaningful to talk about D-branes wrapping arbitrary subspaces (the images $\phi(M)$ from various $K$-cycles) of spacetime, even if the subspaces are not smooth manifolds.

One advantage of our $K$-homological interpretation of D-branes is that it is very much like the standard way of representing fundamental strings. Usually, fundamental strings are characterized, very roughly, by maps from Riemann surfaces (with additional structures) into spacetime. Clearly, this coincides with our interpretation of D-branes as $K$-cycles. It is then intriguing to look for a sensible physical theory that describes the dynamics of D-branes using this idea.

Another advantage of our $K$-homology interpretation of D-branes is that the manifold $M$ in a $K$-cycle $[M,E,\phi]$ is required to be spin$^c$ by definition. If $\phi$ is an embedding, this nicely coincides with the physical requirement that D-branes can only wrap on submanifolds with spin$^c$ structures.

Finally, in this way we put D9-branes (anti-D9-branes) and D-branes with lower dimensions on equal footing. For example, a 9-brane (or an anti-9-brane) is given by a $K$-cycle $[X,E,\phi]$ (or $[X,F,\phi']$) where $\phi$ (or $\phi'$)$:X\rightarrow X$ is a diffeomorphism.
\end{remark}

We have restricted ourselves to the case of type II D-branes. Analogously, one can define topological $K$-homology groups dual to the $K$-theory groups defined using real vector bundles. These groups characterize type I D-branes by a similar argument. Hence, all D-branes in string theory can be properly characterized by suitable topological $K$-homology groups.

\pagebreak

\chapter{Summary and Discussion}

The close relationship and interplay between string theory and pure mathematics has flourished over the last several decades, once again carrying on the old tradition of mutually beneficial development of physics and mathematics. On one hand, the application of pure mathematics to string theory has provided string theorists with numerous powerful tools for computation and model building, as well as fundamental insights about the basic nature of string theory itself. On the other hand, the advances in string theory have generated many deep insights about pure mathematics, as well as many new ways of proving hard results.

In this thesis, we reviewed the interpretation of the RR charges of D-branes in string theory as elements of the topological $K$-theory group of spacetime. The fact that D-branes carry vector bundles on them leads us to $K$-cocycles, while the physically significant process of tachyon condensation leads us to the $K$-theory equivalence relations between $K$-cocycles. This is one of the leading examples of the interplay discussed above.

Furthermore, we proposed a dual interpretation of D-brane charges. More precisely, we proposed a direct characterization of D-branes themselves as elements of the topological $K$-homology group of spacetime. In this approach, D-branes are shown to be characterized by $K$-cycles of spacetime, and, as happened for the equivalence relation defining $K$-theory, the physical tachyon condensation once again lead us to the $K$-homological equivalence relations between $K$-cycles.

One major advantage of using $K$-homology to classify D-branes is that one can directly see the geometry of D-branes from $K$-cycles. For example, given a $K$-cycle $[M,E,\phi]$ which represents a D-brane configuration, one can directly see the D-brane worldvolume, namely $\phi(M)$. If $\phi$ is an embedding in addition, then the gauge bundle of this D-brane configuration is precisely $E$.

Another advantage of $K$-homology is that the tachyon condensation process is more transparent. We have proven that a $K$-cycle $[M,E,\phi]$, where $\phi$ is an embedding, is $K$-homologically equivalent to a $K$-cycle of the form $[X,V- W, id_X]$, where $V$ and $W$ are the gauge bundles on some D9-branes and anti-D9-branes wrapping the entire spacetime $X$. This is precisely the tachyon condensation from the physics point of view.

Our interpretation of D-branes as $K$-cycles coincides with the modern way of characterizing fundamental strings. In modern versions of string theory, fundamental strings are usually characterized by maps from two-dimensional manifolds into spacetime (nonlinear sigma models), with the image of such a map being interpreted as a string worldsheet. In our case, D-branes are given by maps from manifolds of various dimensions into spacetime, with the image of such a map being interpreted as a D-brane worldvolume. This interpretation is especially called for in type IIB string theory, in which the D1-branes (D-strings) have been shown to be able to combine with fundamental strings to form bound states, hence suggesting a similarity between D-strings and fundamental strings. This is very natural in our $K$-homological interpretation, because D-strings are given by maps from two-dimensional manifolds into spacetime, similar to fundamental strings.

Our $K$-homological classification of D-branes suggests a possible understanding of D-branes wrapping subspaces of spacetime that are not necessarily smooth. Quantum field theories are not clearly to be defined on singular spaces, thus preventing us from finding a low energy effective description of such a D-brane configuration on a singular subspace. However, string theory is known to be a generalization of quantum field theories that has shown its ability to overcome several types of singularities of spacetime. Therefore, our $K$-homological classification of D-branes strongly suggests a stronger power of string theory in handling singularities. The detailed physical description is left to future work.


\begin{thebibliography}{199}

\addcontentsline{toc}{chapter}{Bibliography}

\bibitem{asakawa-khomology} T. Asakawa, S. Sugimoto, S. Terashima, \emph{D-branes, Matrix Theory And $K$-homology}, JHEP 0203 (2002) 034, arXiv:hep-th/0108085.

\bibitem {atiyah-ktheory} M. Atiyah, \emph{K-theory}, Advanced Book Classics, Addison-Wesley, 1967.

\bibitem{atiyah} M. Atiyah, \emph{Global Theory Of Elliptic Operators}, Proc. Internat. Conf. on Functional
Analysis and Related Topics, 21-30, University of Tokyo Press, Tokyo, 1969.

\bibitem{blackadar} B. Blackadar, \emph{K-Theory For Operator Algebras}, Cambridge University Press, revised edition, 1998.

\bibitem{bott-peorid} M. Atiyah, R. Bott, \emph{On The Periodicity Theorem For Complex Vector Bundles}, Acta. Math. \textbf{112} (1964), 229-247.

\bibitem{abs} M. Atiyah, R. Bott, A. Shapiro, \emph{Clifford Modules}, Topology \textbf{3}, Suppl. 1 (1964), 3-38.

\bibitem{atiya-hirzebruch} M. Atiyah, F. Hirzebruch, \emph{Riemann-Roch Theorems For Differentiable Manifolds}, Bull. Am. Math. Soc. \textbf{65} (1959) 276-281.

\bibitem{baum-khomology-brane} P. Baum, \emph{$K$-homology And D-branes}, in Superstrings, Geometry, Topology, and $C^*$-algebras, volume 8 of Proc. Sympos. Pure Math. 81-94. Amer. Math. Soc., Providence, R.I., 2009.

\bibitem{baum-douglas} P. Baum, R. Douglas, \emph{K Homology And Index Theory}, in Operator algebras and applications, part I (Kingston, Ont., 1980), volume 38 of Proc. Sympos. Pure Math., 117-173. Amer. Math. Soc., Providence, R.I., 1982

\bibitem{freed-witten-anomaly} D. Freed, E. Witten, \emph{Anomalies In String Theory With D-branes }, Asian J. Math \textbf{3} (1999) 819, arXiv:hep-th/9907189.

\bibitem{harvey-moore-khomology} J. Harvey, G. Moore, \emph{Noncommutative Tachyons And $K$-theory}, J. Math. Phys. \textbf{42} (2001) 2765-2780, arXiv:hep-th/0009030.

\bibitem{horava-ktheory} P. Horava, \emph{Type II D-Branes, K-Theory, And Matrix Theory}, Adv. Theor. Math. Phys. \textbf{2} (1999) 1373; arXiv:hep-th/9812135.

\bibitem{thom} H. Lawson, M. Michelsohn, \emph{Spin Geometry}, Princeton University Press, New Jersey, 1989.

\bibitem{minasian-moore-ktheory} R. Minasian and G. Moore, \emph{K Theory And Ramond-Ramond Charge}, JHEP 9711 (1997) 002, arXiv:hep-th/9710230.

\bibitem{Polchinski-brane} J. Polchinski, \emph{Dirichlet Branes And Ramond-Ramond Charges}, Phys. Rev. Lett. \textbf{75}, (1995) 4724-4727, arXiv:hep-th/9510017.

\bibitem{reis-szabo-khomology} R. Reis, R. Szabo, \emph{Geometric $K$-homology Of Flat D-branes}, Commun. Math. Phys. \textbf{266} (2006) 71-122, arXiv:hep-th/0507043.

\bibitem{szabo} R. Szabo, \emph{Superconnections, Anomalies And Non-BPS Brane Charges}, J. Geom. Phys. \textbf{43} (2002) 241¨C292, arXiv:hep-th/0108043.

\bibitem{witten-anomaly} E. Witten, \emph{Baryons And Branes In Anti-de Sitter Space}, JHEP 9807 (1998) 006, arXiv:hep-th/9805112.

\bibitem{witten-ktheory} E. Witten, \emph{D-Branes And $K$-Theory}, JHEP 9812 (1998) 019, arXiv:hep-th/9810188.







\end{thebibliography}
\end{document}